\documentclass[a4paper,11pt]{article}
\usepackage{amsmath}
\usepackage{amsfonts}
\usepackage{amssymb}
\usepackage{mathrsfs}
\usepackage{amscd}
\usepackage{pb-diagram}
\usepackage{amsthm}
\usepackage{color}
\usepackage[all]{xy}
\usepackage{graphicx}
\usepackage{url}
\usepackage{enumerate}
\usepackage{tikz-cd}
\numberwithin{equation}{section} 
\usepackage{caption}
\DeclareUnicodeCharacter{2212}{-}
\usepackage{tikz}

\usepackage{titlesec}
\titleformat{\subsection}[runin]{\normalsize\bfseries}{\thesubsection}{5pt}{}

\usepackage{tikz}          
\usetikzlibrary{matrix, arrows, decorations.pathmorphing}

\newcommand{\moment}{
\textup{\textbf{J}}
}

\author{Mathieu Molitor
\\
\it \small{Departamento de Matem\'{a}tica, Universidade Federal da Bahia}\\
\it \small{Av. Adhemar de Barros, S/N, Ondina, 40170-110 Salvador, BA, Brazil}\\ 
\small{\it{e-mail:}}\,\,\url{pergame.mathieu@gmail.com}}
\title{Moment polytopes of toric exponential families \\
}
\date{}

\begin{document}

\theoremstyle{definition}
\newtheorem{lemma}{Lemma}[section]
\newtheorem{definition}[lemma]{Definition}
\newtheorem{proposition}[lemma]{Proposition}
\newtheorem{corollary}[lemma]{Corollary}
\newtheorem{theorem}[lemma]{Theorem}
\newtheorem{remark}[lemma]{Remark}
\newtheorem{example}[lemma]{Example}
\newtheorem{setting}[lemma]{Setting}
\bibliographystyle{alpha}

\maketitle 


\begin{abstract}
	We show that the moment polytope of a K\"{a}hler toric manifold, 
	constructed as the torification (in the sense of \cite{molitor-toric}) 
	of an exponential family defined on a finite sample space, is the projection 
	of a higher-dimensional simplex.
\end{abstract}

\section{Introduction}

	Let $\mathcal{E}=\{p_{\theta}:\Omega\to \mathbb{R}\,\,\vert\,\, \theta\in \mathbb{R}^{n}\}$ 
	be an $n$-dimensional exponential family defined on a finite sample space $\Omega=\{x_{0},...,x_{m}\}$ 
	endowed with the counting measure. We assume there exist functions $C,F_{1},...,F_{n}\,:\,\Omega\to \mathbb{R}$ 
	such that $p_{\theta}(x)=\textup{exp}\big\{C(x)+F_{1}(x)\theta_{1}+...+F_{n}(x)\theta_{n}-\psi(\theta)\big\}$ 
	for all $x\in \Omega$ and all $\theta=(\theta_{1},...,\theta_{n})\in \mathbb{R}^{n}$. 
	The function $\psi:\mathbb{R}^{n}\to \mathbb{R}$, known as the \textit{log-partition function}, 
	is determined by the normalization condition $\sum_{x\in \Omega}p_{\theta}(x)=1$. 
	We further assume that the functions $1,F_{1},...,F_{n}$ are linearly independent, ensuring 
	that the map $\mathbb{R}^{n}\to \mathcal{E}$, $\theta\mapsto p_{\theta}$ is bijective.

	Let $h_{F}$ and $\nabla^{(e)}$ be the Fisher metric and exponential connection on $\mathcal{E}$, respectively. 
	It is well-known that $(\mathcal{E},h_{F},\nabla^{(e)})$ is a dually flat manifold \cite{Amari-Nagaoka}. 
	The central assumption of this paper is that $(\mathcal{E},h_{F},\nabla^{(e)})$ is toric in the sense 
	of \cite{molitor-toric} (see Section \ref{nknkneknknfwknk}). 
	This means that we can uniquely associate to $(\mathcal{E},h_{F},\nabla^{(e)})$ 
	a connected, complete and simply connected 
	K\"{a}hler manifold $N$ with real analytic K\"{a}hler metric. This manifold $N$, called \textit{torification} 
	of $(\mathcal{E},h_{F},\nabla^{(e)})$, is endowed with an effective holomorphic and Hamiltonian 
	torus action $\Phi:\mathbb{T}^{n}\times N\to N$ with momentum map 
	$\moment:N\to \mathbb{R}^{n}$ (here $\mathbb{T}^{n}=\mathbb{R}^{n}/\mathbb{Z}^{n}$).

	The main result of this paper is a description of the set $\moment(N)\subset \mathbb{R}^{n}$.
	This object, known as the moment polytope in symplectic geometry, 
	is important in connection with the convexity results of Atiyah \cite{Atiyah} and 
	Guillemin and Sternberg \cite{Guillemin82}, and Delzant's correspondence \cite{Delzant}. 
	To describe our result, we first note that if $S$ is a family of 
	probability functions defined on $\Omega$ (like $\mathcal{E}$), 
	then the convex hull of $S$ within the vector space of real-valued functions on 
	$\Omega$, denoted by $\textup{conv}(S)$, is also formed by probability functions. 
	In particular, $\textup{conv}(S)$ can be regarded as a subset of 
	the hyperplane $H=\{(x_{1},...,x_{m+1})\in \mathbb{R}^{m+1}\,\,\vert\,\, 
	x_{1}+...+x_{m+1}=1\}\subset \mathbb{R}^{m+1}$. We shall say that $S$ is 
	\textit{full} if the topological interior of $\textup{conv}(S)$, regarded as a subspace of $H$, 
	is nonempty. 
\begin{theorem}\label{nfekwnkefnk}
	Suppose $\mathcal{E}$ is full and $\moment:N\to \mathbb{R}^{n}$ is closed 
	(meaning that the image of every closed set is closed). 
	Then there exist an $n\times m$ matrix $T$ with integer entries and 
	$C\in \mathbb{R}^{n}$ such that 
	\begin{eqnarray*}
		\moment(N)=-4\pi T(\Delta_{m})+C, 
	\end{eqnarray*}
	where $\Delta_{m}$ is the $m$-simplex given by $\{(x_{1},...,x_{m})\in \mathbb{R}^{m}\,\,\vert\,\, x_{k}\geq 0\,\,
	\textup{for all $k$ and}\,x_{1}+...+x_{m}\leq 1\}$. 
\end{theorem}
	The theorem imposes strong constraints on the shape of the polytope $\moment(N).$ 
	For instance, the vertices of $\moment(N)$ are necessarily images of 
	vertices of $\Delta_{m}$ under the affine map $-4\pi T+C$, which leads to the following result.
\begin{corollary}
	The difference of two vertices of $\moment(N)$ lies in $4\pi \mathbb{Z}^{n}$. 
\end{corollary}

\begin{example}\label{nekdnwknfkdn}
	Let $\mathbb{P}_{1}(c)$ be the complex projective space of complex dimension $1$ 
	and holomorphic sectional curvature $c>0$, equipped with the torus action 
	$[t]\cdot [z_{0},z_{1}]=[e^{2i\pi t}z_{0},z_{1}]$ (homogeneous coordinates). This action is 
	Hamiltonian with momentum map $\moment_{c}:\mathbb{P}_{1}(c)\to \mathbb{R}$ given 
	by $\moment_{c}([z_{0},z_{1}])=-\tfrac{4\pi}{c} \tfrac{|z_{0}|^{2}}{|z_{0}|^{2}+|z_{1}|^{2}}$. 
	The moment polytope is the segment $[-\tfrac{4\pi}{c},0]$, whose vertices 
	are $-\tfrac{4\pi}{c}$ and $0$. The difference of these vertices lies in $4\pi\mathbb{Z}$ 
	if and only if $c=\tfrac{1}{n}$ for some integer $n\geq 1$. Therefore, for 
	$\mathbb{P}_{1}(c)$ to be realized as the torification of a full exponential family, it is 
	necessary that $c$ is of the form $c=\tfrac{1}{n}$ for some integer $n\geq 1$. 
	This occurs, for example, when the exponential family consists of all binomial distributions 
	$p(k)=\binom{n}{k}q^{k}(1-q)^{n-k}$, $q\in (0, 1)$, defined over the sample space $\Omega=\{0, 1, ..., n\}$, 
	which corresponds to $\mathbb{P}_{1}(\tfrac{1}{n})$ (see \cite{molitor-toric}). 
\end{example}	

	The representation of the moment polytope $\moment(N)$ as a projection of a higher-dimensional 
	simplex has been observed in the context of algebraic geometry and toric varieties arising from fans 
	(see, e.g., \cite{Audin}, Section VII.2.c). This algebraic approach relies on lifting fan morphisms 
	to analytic maps between toric varieties, which, under suitable conditions, allows a toric variety 
	to be analytically mapped into a complex projective space $\mathbb{P}^{m}$, 
	whose moment polytope is precisely the standard simplex $\Delta_{m}$. 
	Our approach follows a similar path and uses 
	a lifting procedure for affine isometric maps between toric dually flat manifolds, 
	developed in \cite{molitor-toric}. This lifting procedure implies that the 
	torification $N$ of an exponential family $\mathcal{E}$ defined on a finite sample space $\Omega$ can be 
	immersed holomorphically, isometrically and equivariantly into a complex projective space, 
	similarly to the algebraic case. A significant departure from the algebraic proof, however, 
	is our utilization of the properties of the marginal polytope of $\mathcal{E}$ (see Section \ref{ncekndwknkn}), 
	which essentially allows us to replace $\mathcal{E}$ with the set $\mathcal{P}(\Omega)$ of all probability 
	functions on $\Omega$ (see Lemma \ref{ndkknendkdnknk}), simplifying our analysis.
	
	To provide the necessary context, Section \ref{nknkneknknfwknk} briefly reviews the 
	torification construction, and Section \ref{ncekndwknkn} presents the proof.
	To clarify our notation, we define a Lie group action $G\times M\to M$ on a symplectic manifold 
	$(M,\omega)$ to be \textit{Hamiltonian} if the action preserves $\omega$ and if 
	there is a map $\moment:M\to \mathfrak{g}^{*}=\textup{Lie}(G)^{*}$, 
	called \textit{momentum map}, that is equivariant with respect to $\Phi$ and the coadjoint action, 
	and such that $\omega(\xi_{M},u)=d\moment^{\xi}(u)$ 
	for all $\xi\in \mathfrak{g}$ and $u\in TM$. Here $\xi_{M}(p)=\tfrac{d}{dt}\big\vert_{0}\Phi(\textup{exp}(t\xi),p)$ 
	is the fundamental vector field associated to $\xi$ and $\moment^{\xi}:M\to \mathbb{R}$ is the function defined by 
	$\moment^{\xi}(p)=\moment(p)(\xi)$. If $G=\mathbb{T}^{n}=\mathbb{R}^{n}/\mathbb{Z}^{n}$ is an $n$-dimensional 
	real torus, then we identify the Lie algebra of $\mathbb{T}^{n}$ with $\mathbb{R}^{n}$ via the derivative at 
	zero of the quotient map $\mathbb{R}^{n}\to \mathbb{R}^{n}/\mathbb{Z}^{n}$. We identify 
	$\mathbb{R}^{n}$ with its dual via the Euclidean metric. Thus, we regard a momentum 
	map for a torus action $\mathbb{T}^{n}\times M\to M$ as a map from $M$ to $\mathbb{R}^{n}$. 
%
%
%
%
%
%
%
%

\section{Torification of dually flat manifolds}\label{nknkneknknfwknk}
	In this section, we briefly review the concept of torification, which is used throughout this paper. This concept 
	is a combination of two ingredients: (1) \textit{Dombrowski's construction}, 
	which implies that the tangent bundle of a dually flat manifold is naturally a K\"{a}hler manifold \cite{Dombrowski}, 
	and (2) \textit{parallel lattices}, which are used to implement torus actions. 
	The material is mostly taken from \cite{molitor-toric}. 

\subsection{Dombrowski's construction.}\label{nkwwnkwnknknfkn}
	Let $M$	be a connected manifold of dimension $n$, endowed with a Riemannian metric $h$ 
	and affine connection $\nabla$ ($\nabla$ is not necessarily 
	the Levi-Civita connection). The \textit{dual connection} of $\nabla$, denoted by $\nabla^{*}$, is 
	the only connection satisfying $X(h(Y,Z))=h(\nabla_{X}Y,Z)+h(Y,\nabla^{*}_{X}Z)$ for all vector fields $X,Y,Z$ on $M$. 
	When both $\nabla$ and $\nabla^{*}$ are flat (i.e., the curvature tensor and torsion are zero), 
	we say that the triple $(M,h,\nabla)$ is a \textit{dually flat manifold}. 

	Let $\pi:TM\to M$ denote the canonical projection. Given a local coordinate system $(x_{1},...,x_{n})$ 
	on $U\subseteq M$, we can define a coordinate system $(q,r)=(q_{1},...,q_{n},r_{1},...,r_{n})$ 
	on $\pi^{-1}(U)\subseteq TM$ by letting $(q,r)(\sum_{j=1}^{n}a_{j}\tfrac{\partial}{\partial x_{j}}\big\vert_{p})=
	(x_{1}(p),...,x_{n}(p),a_{1},...,a_{n})$, where $p\in M$ and $a_{1},...,a_{n}\in \mathbb{R}$. 
	Write $(z_{1},...,z_{n})=(q_{1}+ir_{1},...,q_{n}+ir_{n})$, where $i=\sqrt{-1}$. When $\nabla$ is flat, Dombrowski 
	\cite{Dombrowski} showed that the family of complex coordinate systems $(z_{1},...,z_{n})$ 
	on $TM$ (obtained from affine coordinates on $M$) form a holomorphic atlas on $TM$. Thus, when $\nabla$ is flat, 
	$TM$ is naturally a complex manifold. If in addition $\nabla^{*}$ is flat, then $TM$ has a natural 
	K\"{a}hler metric $g$ whose local expression in the coordinates $(q,r)$ is given by $g(q,r)=
	\big[\begin{smallmatrix}
		h(x)  &  0\\
		0     &   h(x)
	\end{smallmatrix}
	\big]$, where $h(x)$ is the matrix representation of $h$ in the affine coordinates $x=(x_{1},...,x_{n})$. 
	It follows that the tangent bundle of a dually flat manifold is naturally a K\"{a}hler manifold. In this paper, 
	we will refer to this K\"{a}hler structure as the \textit{K\"{a}hler structure associated to Dombrowski's 
	construction}. 

\subsection{Parallel lattices.}\label{nekwnkefkwnkk}

	Let $(M,h,\nabla)$ be a connected dually flat manifold of dimension $n$. 
	Recall that a \textit{frame} on $M$ is an ordered $n$-tuple $(E_{1},...,E_{n})$ of vector fields $E_{i}$
	on $M$ such that $E_{1}(p),...,E_{n}(p)$ form a basis for $T_{p}M$ for each $p\in M$. 
	We will often use the notations $E$ and $(E_{i})$ interchangeably to denote a frame $(E_{1},...,E_{n})$.
	Let $\textup{Fr}(M)$ be the set of frames on $M$, and let $\textup{Fr}(M,\nabla)$ be the subset of 
	$\textup{Fr}(M)$ consisting of frames $(E_{1},...,E_{n})$ 
	such that $E_{i}$ is parallel with respect to $\nabla$ for all $i=1,...,n$. 

	A set $\mathcal{L}\subset TM$ is said to be a \textbf{parallel lattice} with respect to $\nabla$ if there is 
	a frame $E\in \textup{Fr}(M,\nabla)$ such that $\mathcal{L}=\mathcal{L}(E)$, where $\mathcal{L}(E)\overset{\textup{def}}{=}
	\{k_{1}E_{1}(p)+...+k_{n}E_{n}(p)\,\,\vert\,\,p\in M,\,\,k_{1},...,k_{n}\in \mathbb{Z}\}\subset TM.$
	In this case, we say that $\mathcal{L}$ is \textit{generated} by $E$, and that $E$ is a 
	\textit{generator} for $\mathcal{L}$. 

	Given a parallel lattice $\mathcal{L}\subset TM$ with respect to $\nabla$ generated by $E=(E_{i})$,
	we will denote by $\Gamma(\mathcal{L})$ the set of transformations of $TM$ of the form 
	$u\mapsto u+k_{1}E_{1}+...+k_{n}E_{n}$, where $u\in TM$ and $k_{1},...,k_{n}\in \mathbb{Z}$. The group
	$\Gamma(\mathcal{L})$ is independent of the choice of $E$ and is isomorphic to $\mathbb{Z}^{n}$. 
	If $\mathcal{L}\subset TM$ is a parallel lattice generated by $E=(E_{i})$, then 
	the map $M\times \mathbb{R}^{n}\to TM$, given by $(p,(u_{1},...,u_{n}))\mapsto \sum_{i=1}^{n}u_{i}E_{i}(p)$, 
	is a global trivialization. In therms of this trivialization, the action of 
	$\Gamma(\mathcal{L})\cong \mathbb{Z}^{n}$ on $TM$ is simply $k\cdot (p,u)=
	(p,u+k)$, where $k\in \mathbb{Z}^{n}$, $p\in M$ and $u\in \mathbb{R}^{n}$. Thus the action 
	of $\Gamma(\mathcal{L})$ on $TM$ is free and proper, and the orbit space 
	$TM/\Gamma(\mathcal{L})$ is a manifold diffeomorphic to the Cartesian product 
	$M\times \mathbb{T}^{n}$, where $\mathbb{T}^{n}=\mathbb{R}^{n}/\mathbb{Z}^{n}$ 
	is the real torus of dimension $n$. Because of this, we call the quotient manifold 
	$\textup{Tub}(\mathcal{L}):=TM/\Gamma(\mathcal{L})$ the \textbf{tube} associated to $\mathcal{L}$, 
	and call the corresponding projection $tub_{\mathcal{L}}:TM\to \textup{Tub}(\mathcal{L})$ the 
	\textbf{tubular map}. 

	The tubular map $tub_{\mathcal{L}}$ is a covering map whose Deck transformation group is $\Gamma(\mathcal{L})$. 
	If $\pi:TM\to M$ is the canonical projection, then $\pi\circ \gamma=\pi$ for every 
	$\gamma\in \Gamma(\mathcal{L})$ and hence there is a surjective submersion 
	$\pi_{\mathcal{L}}:\textup{Tub}(\mathcal{L})\to M$ such that $\pi= \pi_{\mathcal{L}}\circ tub_{\mathcal{L}}$. 
	
	Now, given a generator $E=(E_{i})$, we define an action $\Phi_{E}$ of the torus $\mathbb{T}^{n}$ on 
	$\textup{Tub}(\mathcal{L})$ by 
	\begin{eqnarray}\label{ncekdnwkefnekn}
		\Phi_{E}([t],tub_{\mathcal{L}}(u)\big)= tub_{\mathcal{L}}\big(u+t_{1}E_{1}+...+t_{n}E_{n}\big),
	\end{eqnarray}
	where $[t]$ denotes the equivalence class of 
	$t=(t_{1},...,t_{n})\in \mathbb{R}^{n}$ in $\mathbb{T}^{n}=\mathbb{R}^{n}/\mathbb{Z}^{n}$. 
	The map $\pi_{\mathcal{L}}:\textup{Tub}(\mathcal{L})\to M$, together with the torus action $\Phi_{E}$, is then 
	a left principal fiber bundle over $M$ with structure group $\mathbb{T}^{n}$. 
	
	Analytically, the tube $\textup{Tub}(\mathcal{L})=TM/\Gamma(\mathcal{L})$ inherits a K\"{a}hler manifold structure 
	from the K\"{a}hler structure on $TM$ derived from Dombrowski's construction.
	This is due to the fact that each $\gamma\in \Gamma(\mathcal{L})$ is a holomorphic isometry.
	Consequently, the tubular map becomes a K\"{a}hler covering map, meaning it is a holomorphic and isometric covering map.
	Furthermore, the projection $\pi_{\mathcal{L}}:\textup{Tub}(\mathcal{L})\to M$ is a Riemannian submersion, 
	and for every $a\in \mathbb{T}^{n}$, the corresponding transformation 
	$(\Phi_{E})_{a}:\textup{Tub}(\mathcal{L})\to \textup{Tub}(\mathcal{L})$ is a holomorphic isometry. 

\subsection{Torification.}\label{nfkenkfejdefdknkn} 
	Let $(M,h,\nabla)$ be a connected dually flat manifold of dimension $n$ and $N$ a connected K\"{a}hler 
	manifold of complex dimension $n$, equipped with an effective holomorphic and isometric torus 
	action $\Phi:\mathbb{T}^{n}\times N\to N$. Let $N^{\circ}$ denote the set of points $p\in N$ 
	where $\Phi$ is free, that is, $N^{\circ}=\{p\in N\,\,\big\vert\,\,\Phi(a,p)=e\,\,\Rightarrow\,\,a=e\}$. 
	It follows from \cite[Corollary B.48]{Guillemin} that $N^{\circ}$ is a $\mathbb{T}^{n}$-invariant connected open 
	dense subset of $N$.

	Given a parallel lattice $\mathcal{L}\subset TM$, we will say that a map $F:\textup{Tub}(\mathcal{L})\to N^{\circ}$ 
	is \textit{equivariant} if there exists a generator $E$ for $\mathcal{L}$ such that 
	$F\circ (\Phi_{E})_{a}=\Phi_{a}\circ F$ for all $a\in \mathbb{T}^{n}$ (see Section \ref{nekwnkefkwnkk} for the notation).
\begin{definition}[\textbf{Torification}]
	We shall say that $N$ is a \textit{torification} of $M$ 
	if there exist a parallel lattice $\mathcal{L}\subset TM$ with respect to 
	$\nabla$ and an equivariant holomorphic and isometric diffeomorphism $F:\textup{Tub}(\mathcal{L})\to N^{\circ}.$ 
\end{definition}

	By abuse of language, we will often say that the torus action $\Phi:\mathbb{T}^{n}\times N\to N$ is a torification of $M$. 
%

	We shall say that a K\"{a}hler manifold $N$ is \textbf{regular} if it is connected, 
	simply connected, complete and if the K\"{a}hler metric is real analytic. 
	A torification $\Phi:\mathbb{T}^{n}\times N\to N$ is said to be \textit{regular} if $N$ is regular. 
	In this paper, we primarily focus on regular torifications, because they are essentially unique, 
	as we explain below.
	
	Two torifications $\Phi:\mathbb{T}^{n}\times N\to N$ and $\Phi':\mathbb{T}^{n}\times N'\to N'$ 
	of the same connected dually flat manifold $(M,h,\nabla)$ are said to be \textit{equivalent} if there exists 
	a K\"{a}hler isomorphism $f:N\to N'$ and a Lie group isomorphism $\rho:\mathbb{T}^{n}\to \mathbb{T}^{n}$ such that 
	$f\circ \Phi_{a}=\Phi'_{\rho(a)}\circ f$ for all $a\in \mathbb{T}^{n}$.

\begin{theorem}[\textbf{Equivalence of regular torifications}]\label{newdnkekfwndknk}
	Regular torifications of a connected dually flat manifold $(M,h,\nabla)$ are equivalent. 
\end{theorem}

	A connected dually flat manifold $(M,h,\nabla)$ is said to be \textbf{toric} if it has a regular torification $N$. 
	In this case, we will often refer to $N$ as ``the regular torification of $M$", keeping in mind that it is 
	defined up to equivariant K\"{a}hler isomorphisms and reparametrizations of the torus. 


\section{Proof of Theorem \ref{nfekwnkefnk}}\label{ncekndwknkn}

	Let the setup be as described at the beginning of this article. We use $\mathcal{P}(\Omega)$ 
	to denote the set of all probability functions on $\Omega=\{x_{0},...,x_{m}\}$. Thus, if 
	$p\in \mathcal{P}(\Omega)$, then $p(x)\geq 0$ for all $x\in \Omega$ and $\sum_{x\in \Omega}p(x)=1$.
	We use $\mathbb{E}_{p}(X)=\sum_{x\in \Omega}X(x)p(x)$ to denote the expectation of the 
	random variable $X:\Omega\to \mathbb{R}$ with respect to $p\in \mathcal{P}(\Omega)$.

	Recall that the elements of $\mathcal{E}$ are of the form 
	$p_{\theta}(x)=\textup{exp}\big\{C(x)+F_{1}(x)\theta_{1}+...+F_{n}(x)\theta_{n}-\psi(\theta)\big\}$, 
	where $x\in \Omega$, $\theta=(\theta_{1},...,\theta_{n})\in \mathbb{R}^{n}$ and $C,F_{1},...,F_{n}$ are 
	functions on $\Omega$. Let $F:\Omega\to \mathbb{R}^{n}$ be the map defined by $F(x)=(F_{1}(x),...,F_{n}(x))$. 
	The proof of Theorem \ref{nfekwnkefnk} heavily relies on the properties of 
	the convex hull of the finite set $\{F(x)\in \mathbb{R}^{n}\,\,\vert\,\,
	x\in \Omega\}$, 
	which we will denote by $\mathcal{M}.$ In the statistical litterature, $\mathcal{M}$ is known 
	as the \textit{marginal polytope}; it is a convex polytope (by definition) 
	with the following properties (see \cite{Martin}):
	\begin{itemize}
	\item $\mathcal{M}=\{(\mathbb{E}_{p}(F_{1}),...,\mathbb{E}_{p}(F_{n}))\in 
		\mathbb{R}^{n}\,\,\big\vert\,\,p\in \mathcal{P}(\Omega)\}\subset \mathbb{R}^{n}$.
	\item The topological interior of $\mathcal{M}$ is given by $\mathcal{M}^{\circ}=
		\{(\mathbb{E}_{p}(F_{1}),...,\mathbb{E}_{p}(F_{n}))\in 
		\mathbb{R}^{n}\,\,\big\vert\,\,p\in \mathcal{E}\}$. 
	\end{itemize}
	We now proceed with the proof of Theorem \ref{nfekwnkefnk}, which we divide into two parts.

\subsection{Momentum map and marginal polytope.} Because $N$ is the torification of the exponential 
	family $\mathcal{E}$, the momentum map $\moment:N\to \mathbb{R}^{n}$ has a very specific form on the set $N^{\circ}$ 
	of points in $N$ where the torus action $\Phi$ is free (see \cite{molitor-toric,molitor-spectral}). 
	Specifically, there is a surjective Riemannian submersion $\kappa:N^{\circ}\to \mathcal{E}$ 
	and a bijective affine map $A:\mathbb{R}^{n}\to \mathbb{R}^{n}$ such that on $N^{\circ}$, 
	\begin{eqnarray}\label{nfekdckdkjnwkenfk}
		\moment(x)=A(\mathbb{E}_{\kappa(x)}(F_{1}),...,\mathbb{E}_{\kappa(x)}(F_{n})).
	\end{eqnarray}
	Since $\kappa$ is surjective, we have $\moment(N^{\circ})=A(\mathcal{M}^{\circ})$, where 
	$\mathcal{M}^{\circ}$ is the topological interior of the marginal polytope $\mathcal{M}$. 

%
%
%
%
%

	For the next lemma, recall that $\moment$ is closed by hypothesis. This implies that
	$\moment(\overline{X})=\overline{\moment(X)}$ for all subset $X$ of $N$, 
	where $\overline{X}$ denotes the closure of $X$ in $N$.

\begin{lemma}\label{ndkknendkdnknk}
	We have $\moment(N)=A(\mathcal{M})$. 
\end{lemma}
\begin{proof}
	Since $N^{\circ}$ is dense in $N$ and $\moment$ is closed, we have 
	$\moment(N)=\moment(\overline{N^{\circ}})=\overline{\moment(N^{\circ})}$. 
	From the preceding discussion, $\moment(N^{\circ})=A(\mathcal{M}^{\circ})$, 
	and because affine transformations are homeomorphisms, $A(\mathcal{M}^{\circ})=(A(\mathcal{M}))^{\circ}$. 
	Consequently, $\moment(N)=\overline{(A(\mathcal{M}))^{\circ}}$. 
	A crucial observation is that $A(\mathcal{M})$ is a closed convex set with a nonempty topological interior, 
	which implies that $\overline{(A(\mathcal{M}))^{\circ}}=A(\mathcal{M})$ 
	(see \cite{Rockafellar}, Theorem 6.3). Thus $\moment(N)=A(\mathcal{M})$.
\end{proof}

\subsection{K\"{a}hler immersion into $\mathbb{P}^{m}$.} 
	A second crucial element in the proof of Theorem \ref{nfekwnkefnk} is the fact that $N$ can be 
	immersed into a complex projective space. This follows from 
	a lifting procedure for affine isometric maps between toric dually flat manifolds, 
	developed in \cite{molitor-toric}. In the setup of this article, this lifting procedure 
	implies the existence of a holomorphic and isometric immersion $f:N\to \mathbb{P}^{m}$, where 
	$\mathbb{P}^{m}$ is the complex projective space of complex dimension $m$ and holomorphic sectional curvature 
	$c=1$, with the following properties:
	\begin{itemize}
	\item $f$ is equivariant in the sense that there exists 
		a Lie group homomorphism $\rho:\mathbb{T}^{n}\to \mathbb{T}^{m}$, with finite kernel, such that 
		$f\circ \Phi_{a}=\Phi'_{\rho(a)}\circ f$ for all $a\in \mathbb{T}^{n}$, where $\Phi'$ is the 
		action of $\mathbb{T}^{m}=\mathbb{R}^{m}/\mathbb{Z}^{m}$ on $\mathbb{P}^{m}$ defined by 
		$\Phi'([t],[z_{1},...,z_{m+1}])=[e^{2i\pi t_{1}}z_{1},...,e^{2i\pi t_{m}}z_{m},z_{m+1}]$ 
		(homogeneous coordinates).
	\item $K\circ f=\kappa$ on $N^{\circ}$, where $\kappa:N^{\circ}\to \mathcal{E}$ is the Riemannian submersion 
		described in the preceding section and $K:\mathbb{P}^{m}\to \mathcal{P}(\Omega)$ is the 
		surjective map defined by 
		$K([z_{1},...,z_{m+1}])(x_{k})=\tfrac{|z_{k+1}|^{2}}{|z_{1}|^{2}+...+|z_{m+1}|^{2}}$, $x_{k}\in \Omega$.
	\end{itemize}
	The $m$-simplex $\Delta_{m}\subset \mathbb{R}^{m}$, which plays a key role in Theorem \ref{nfekwnkefnk}, originates 
	from the torus action $\Phi'$ on $\mathbb{P}^{m}$. Indeed, $\Phi'$ is Hamiltonian with momentum map 
	$\moment':\mathbb{P}^{m}\to \mathbb{R}^{m}$ given by $\moment'=\alpha\circ K$, 
	where $\alpha:\mathcal{P}(\Omega)\to \mathbb{R}^{m}$ is defined by 
	$\alpha(p)=-4\pi (p(x_{0}),...,p(x_{m-1}))$, and thus $\moment'(\mathbb{P}^{m})=\alpha(\mathcal{P}(\Omega))=-4\pi \Delta_{m}$.

	Let $T:\mathbb{R}^{m}\to \mathbb{R}^{n}$ be the unique linear map that satisfies 
	$\langle T(u),v\rangle=\langle u,\rho_{*_{e}}(v)\rangle$ for all $u\in \mathbb{R}^{m}$ and 
	$v\in \mathbb{R}^{n}$, where $\langle\,,\, \rangle$ denotes the Euclidean pairing on both 
	$\mathbb{R}^{n}$ and $\mathbb{R}^{m}$. Here we identify the Lie algebra of 
	$\mathbb{T}^{n}$ with $\mathbb{R}^{n}$ via the derivative at 
	zero of the quotient map $\mathbb{R}^{n}\to \mathbb{R}^{n}/\mathbb{Z}^{n}=\mathbb{T}^{n}$, 
	which allows us to regard $\rho_{*_{e}}$ as a map from $\mathbb{R}^{n}$ to $\mathbb{R}^{m}$. 

\begin{lemma}
	The matrix representation of $T$ with respect to the standard bases of $\mathbb{R}^{m}$ and $\mathbb{R}^{n}$ 
	has integer entries.
\end{lemma}
	The lemma follows from basic facts on Lie group theory that we briefly recall now. 
	If $\rho:G_{1}\to G_{2}$ is a Lie group homomorphism, then $\rho\circ \exp_{G_{1}}=\textup{exp}_{G_{2}}\circ \rho_{*_{e}}$, 
	where $\textup{exp}_{G_{i}}:\textup{Lie}(G_{i})\to G_{i}$ is the exponential map of $G_{i}$ 
	(see, e.g., \cite{kolk}, Lemma 1.5.1). When $G_{1}=\mathbb{T}^{n}$, 
	the exponential map is just the quotient map 
	$\textup{Lie}(\mathbb{T}^{n})=\mathbb{R}^{n}\to \mathbb{T}^{n}$, $t\mapsto [t]$, and thus, if $\rho:\mathbb{T}^{n}\to \mathbb{T}^{m}$ 
	is a Lie group homomorphism, then $\rho([t])=[\rho_{*_{e}}t]$ for all $t\in \mathbb{R}^{n}$. 
	Clearly, this forces the matrix representation of $\rho_{*_{e}}:\mathbb{R}^{n}\to \mathbb{R}^{m}$ with respect to the 
	standard bases to have integer entries. The lemma is then immediate. 

\begin{lemma}\label{nfeknwkenfk}
	There exists $C\in \mathbb{R}^{n}$ such that $\moment=T\circ \moment'\circ f+C$. 
\end{lemma}
\begin{proof}
	It is straightforward to verify that $T\circ \moment'\circ f$ 
	is a momentum map for the action $\Phi$. Consequently, it differs from $\moment$ by a constant.
\end{proof}
	
	Let $A$ be the affine transformation of $\mathbb{R}^{n}$ defined in the preceding section. 
	The following lemma is the key technical result that connects $\moment(N)$ with the simplex $\Delta_{m}$. 
	
\begin{lemma}\label{ncdnkefnkenkn}
	For every $p\in \textup{conv}(\mathcal{E})$, 
	\begin{eqnarray}\label{nfeknkenfknk}
		(T\circ \alpha)(p)+C=A(\mathbb{E}_{p}(F_{1}),...,\mathbb{E}_{p}(F_{n})).
	\end{eqnarray}
	If, in addition, $\mathcal{E}$ is full, then the equality holds for every $p\in \mathcal{P}(\Omega)$. 
\end{lemma}
\begin{proof}
	It follows directly from \eqref{nfekdckdkjnwkenfk}, Lemma \ref{nfeknwkenfk} and the formula 
	$\moment'\circ f= \alpha\circ K\circ f=\alpha \circ \kappa$ that 
	\eqref{nfeknkenfknk} holds for all $p\in \mathcal{E}$. The fact that it also holds for all 
	$p\in \textup{conv}(\mathcal{E})$ is then immediate. If $\mathcal{E}$ is full, then we can interpret 
	\eqref{nfeknkenfknk} as an equality of two continuous functions on a convex polytope, real-analytic on its interior, 
	which coincide on an open set. Therefore they coincide everywhere.
\end{proof}

	We now proceed with the proof of Theorem \ref{nfekwnkefnk}. Suppose $\mathcal{E}$ is full. 
	Since $\alpha(\mathcal{P}(\Omega))=-4\pi \Delta_{m}$ by the preceding discussion, 
	Lemma \ref{ncdnkefnkenkn} implies that $-4\pi T(\Delta_{m})+C=A(\mathcal{M})$, 
	where $\mathcal{M}$ is the marginal polytope. By Lemma \ref{ndkknendkdnknk}, $A(\mathcal{M})=\moment(N)$. 
	Therefore, $-4\pi T(\Delta_{m})+C=\moment(N)$. This completes the proof of Theorem \ref{nfekwnkefnk}.

\begin{example}
	If $N=\mathbb{P}_{1}(\tfrac{1}{n})$, as discussed in Example \ref{nekdnwknfkdn}, is the complex projective space, 
	regarded as the regular torification of the family of binomial distributions on $\Omega=\{0,...,n\}$, 
	then the immersion from $N$ into a complex projective space (obtained from the lifting procedure discussed above) 
	is the Veronese embedding $f:\mathbb{P}_{1}(\tfrac{1}{n})\to \mathbb{P}_{n}(1)$, given by (see \cite{molitor-toric}):
	\begin{eqnarray*}
		f([z_{0},z_{1}])=\text{$\big[z_{0}^{n},...,\textstyle \binom{n}{k}^{1/2}z_{0}^{n-k}z_{1}^{k},...,z_{1}^{n}\big].$}
	\end{eqnarray*}
	The corresponding group homomorphism $\rho:\mathbb{T}\to \mathbb{T}^{n}$ is given by $\rho([t])=[nt,(n-1)t,...,2t,t]$, 
	$t\in \mathbb{R}$. Thus, $T=[n\,n-1\,...\,1]$. 
\end{example}

\begin{footnotesize}\bibliography{bibtex}\end{footnotesize}
\end{document}